\documentclass[11pt]{article}
\usepackage{graphicx}
\usepackage{amsmath,amsfonts,amssymb,amsthm}
\usepackage{bbm,mathrsfs}
\usepackage{bm}
\title{Conic support measures}
\author{Rolf Schneider}
\date{}
\newcommand{\Sd}{{\mathbb S}^{d-1}}
\newcommand{\R}{{\mathbb R}}

\newcommand{\C}{{\mathcal C}}

\newcommand{\bP}{{\mathbb P}}
\newcommand{\cP}{{\mathcal P}}

\newcommand{\Rd}{{\mathbb R}^d}
\newcommand{\N}{{\mathbb N}}

\newcommand{\B}{\mathcal{B}}
\newcommand{\D}{{\rm d}}
\newcommand{\F}{{\mathcal F}}

\newcommand{\PC}{{\mathcal PC}}

\newcommand{\bE}{{\mathbb E}\,}

\newcommand{\g}{{\bf g}}
\newcommand{\BP}{{\mathbb P}}

  \newcommand{\fed}{\,\rule{.1mm}{.20cm}\rule{.20cm}{.1mm}\,}

\newtheorem{theorem}{Theorem}[section]
\newtheorem{lemma}{Lemma}[section]

\newtheorem{definition}{Definition}[section]

\begin{document}
\maketitle

\begin{abstract}
The conic support measures localize the conic intrinsic volumes of closed convex cones in the same way as the support measures of convex bodies localize the intrinsic volumes of convex bodies. In this note, we extend the `Master Steiner formula' of McCoy and Tropp, which involves conic intrinsic volumes, to conic support measures. Then we prove H\"{o}lder continuity of the conic support measures with respect to the angular Hausdorff metric on convex cones and a metric on conic support measures which metrizes the weak convergence.\\[1mm]
\end{abstract}

\section{Introduction}\label{sec1}

It is well known that the intrinsic volumes, which play a prominent role in the Brunn--Minkowski theory and in the integral geometry of convex bodies, have local versions in the form of measures. One type of these, the area measures of convex bodies, was introduced independently by Aleksandrov \cite{Ale37a} and by Fenchel and Jessen \cite{FJ38}. The other series, the curvature measures, was defined and thoroughly studied by Federer \cite{Fed59} (more generally, for sets of positive reach). A common generalization, the support measures, of which area and curvature measures are marginal measures, was apparently first introduced (though not under this name) and used in \cite{Sch79}. Support measures have their name from the fact that they are defined on Borel sets of support elements; a support element of a convex body $K$ is a pair $(x,u)$, where $x$ is a boundary point of $K$ and $u$ is an outer unit normal vector of $K$ at $x$. Characterization theorems of the three series of measures are proved in \cite{Sch75a}, \cite{Sch78}, \cite{Gla97}, with integral-geometric applications in \cite[formula (4.3)]{Sch75b}, \cite[formula (7.1)]{Sch78}, \cite[Theorem 3.1]{Gla97}. Integral-geometric formulas for support measures appear in \cite{Sch86}, \cite{Sch88}, \cite{Gla97}. We refer also to the surveys \cite{Sch79b}, \cite{Sch93} on curvature and surface area measures, the surveys \cite{SW93}, \cite{HS02} on integral geometric formulas, and to Chapter 4 in the book \cite{Sch14}.

In spherical integral geometry, the role of the Euclidean intrinsic volumes is taken over by the spherical intrinsic volumes. Their equivalent counterparts for closed convex cones in $\R^d$ are known as conic intrinsic volumes. In spherical space, the local integral geometry of convex bodies, involving spherical analogues of support measures and curvature measures, was developed by Glasauer \cite{Gla95}; a short summary appears in \cite{Gla96}. Parts of this, as far as they are relevant for stochastic geometry, are reproduced in \cite[Section 6.5]{SW08}.

The spherical kinematic formula, together with the spherical Gauss--Bonnet theorem, can be used to answer the following question, in terms of a formula involving conic intrinsic volumes. If $C,D$ are closed convex cones in $\Rd$, and if a uniform random rotation $\bf \Theta$ is applied to $D$, what is the probability that $C$ and ${\bf \Theta}D$ have a non-trivial intersection? This question, mainly for polyhedral cones, and subsequent studies of conic intrinsic volumes, have recently played a central role in various investigations on conic optimization and signal demixing under certain random models (see, e.g. \cite{AB12, AB15, ALMT14, GNP17, MT14a}). This caused renewed interest in spherical integral geometry and was an incentive to reconsider parts of it from the conic viewpoint, to develop new proofs and to extend some results; see \cite{Ame15}, \cite{AL17}, \cite{Sch17}. We mention that also a local version of the conic kinematic formula has proved useful for applications (see \cite{AB15}). It may, therefore, be the right time for a more detailed study of the conic support measures, which are the local versions of the conic intrinsic volumes. In the following, we begin such a study. Our first main aim is to localize the `Master Steiner formula' of McCoy and Tropp \cite{MT14b}; see Theorem \ref{T3.4}. Then we prove a H\"older continuity property of the conic support measures, with respect to the angular Hausdorff metric on the convex cones and a suitable metric on the conic support measures, which metrizes the weak convergence. The result is Theorem \ref{T4.1}.

\section{Preliminaries on convex cones}\label{sec2}

We work in the $d$-dimensional real vector space $\Rd$ ($d\ge 2$), with scalar product $\langle\cdot\,,\cdot\rangle$ and induced norm $\|\cdot\|$. The origin of $\Rd$ is denoted by $o$, and the unit sphere by $\Sd$. The total spherical Lebesgue measure of $\Sd$ is given by $\omega_d =2\pi^{d/2}/\Gamma(d/2)$. By $\C^d$ we denote the set of closed convex cones (with apex $o$) in $\R^d$. Occasionally, we have to exclude the cone $\{o\}$; therefore, we introduce the notation $\C^d_*$ for the set of cones in $\C^d$ of positive dimension. For $C\in \C^d$, the dual cone is defined by
$$ C^\circ = \{x\in\Rd: \langle x,y\rangle\le 0\mbox{ for all }y\in C\}.$$
Writing $C^{\circ\circ}=(C^\circ)^\circ$, we have $C^{\circ\circ}=C$. The {\em nearest point map} or {\em metric projection} $\Pi_C:\Rd\to C$ associates with $x\in \R^d$ the unique point $\Pi_C(x)\in C$ for which 
$$ \|x-\Pi_C(x)\| =\min\{\|x-y\|:y\in C\}.$$
It satisfies $\Pi_C(\lambda x)=\lambda\Pi_C(x)$ for $x\in\Rd$ and $\lambda\ge 0$.

For two vectors $x,y\in\Rd$, we denote by $d_a(x,y)$ the angle between them, thus
$$ d_a(x,y)= \arccos\left\langle \frac{x}{\|x\|},\frac{y}{\|y\|}\right\rangle \quad\mbox{if }x,y\not=o.$$
Restricted to unit vectors, this yields the usual geodesic metric on $\Sd$. The definition is supplemented by $d_a(o,o)=0$ and $d_a(x,o)=d_a(o,x)=\pi/2$ for $x\in\Rd\setminus\{o\}$. Let $C\in \C^d_*$. The {\em angular distance} $d_a(x,C)$ of a point $x\in\Rd$ from $C$ is defined by
$$ d_a(x,C):=\arccos\frac{\|\Pi_C(x)\|}{\|x\|} \quad\mbox{if } x\not=o,$$
and $d_a(o,C)=\pi/2$. Thus,
$$ d_a(x,C) =\left\{\begin{array}{ll} \min\{d_a(x,y): y\in C\} & \mbox{if } x\notin C^\circ,\\[1mm] \pi/2 & \mbox{if } x\in C^\circ.\end{array}\right.$$ 
From the Moreau decomposition (Moreau \cite{Mor62})
$$\Pi_C(x)+\Pi_{C^\circ}(x)=x,\qquad\langle \Pi_C(x),\Pi_{C^\circ}(x)\rangle =0$$ 
for $C\in\C^d$ and $x\in\Rd$ it follows that
\begin{equation}\label{2.1}
d_a(x,C) +d_a(x,C^\circ)= \pi/2\quad\mbox{for } x\in\Rd\setminus(C\cup C^\circ).
\end{equation}

For $C\in\C^d_*$ and $\lambda \ge 0$, we define the {\em angular parallel set} of $C$ at distance $\lambda$ by
$$ C^a_\lambda := \{x\in\R^d: d_a(x,C)\le\lambda\}.$$
The {\em angular Hausdorff distance} of $C,D\in\C^d_*$ is then defined 
by
$$ \delta_a(C,D):= \min\{\lambda\ge 0: C\subseteq D^a_\lambda,\; D\subseteq C^a_\lambda\}.$$
Similarly as for the usual Hausdorff metric, one shows that this defines a metric on $\C^d_*$. With respect to this metric, polarity is a local isometry.

\begin{lemma}\label{L2.1}
If $C,D\in\C^d_*$ are cones $\not=\R^d$ with $\delta_a(C,D) <\pi/2$, then
$$ \delta_a(C^\circ,D^\circ)=\delta_a(C,D).$$
\end{lemma}

The proof given by Glasauer \cite[Hilfssatz 2.2]{Gla95} for the spherical counterpart is easily translated to the conic situation.

The following two lemmas estimate the distance between the projections $\Pi_C(x)$ and $\Pi_D(x)$, either Euclidean or angular, if $C$ and $D$ have small angular Hausdorff distance.

\begin{lemma}\label{L2.1a}
Let $C,D \in \C^d_*$ and $x \in \R^d$. Then
$$ \|\Pi_C(x)-\Pi_D(x)\| \le \|x\|  \sqrt{10 \delta_a(C,D)}.$$
\end{lemma}

\begin{proof}
Let $B_{\|x\|}$ denote the ball with center $o$ and radius $\|x\|$, and let $K:=C\cap B_{\|x\|}$ and $L:= D\cap B_{\|x\|}$. We have $\Pi_C(x)\in B_{\|x\|}$ and hence $\Pi_C(x)=p(K,x)$, where $p(K,\cdot)$ denotes the nearest-point map of the convex body $K$ (as in \cite[Sec. 1.2]{Sch14}); similarly $\Pi_D(x)= p(L,x)$. Elementary geometry together with a rough estimate shows that the Euclidean Hausdorff distance of $K$ and $L$ satisfies $\delta(K,L) \le \delta_a(C,D)\|x\|$. Now \cite[Lem. 1.8.11]{Sch14} yields
$$ \|\Pi_C(x)-\Pi_D(x)\|= \|p(K,x)-p(L,x)\| \le \sqrt{10\|x\|\delta(K,L)} \le \|x\|\sqrt{10 \delta_a(C,D)},$$
as asserted. 
\end{proof}

\begin{lemma}\label{L2.2}
Let $C,D\in\C^d_*$ and $x\in\R^d$. Suppose there exists $\varepsilon>0$ with $\delta_a(C,D)\le\pi/2-\varepsilon$, $d_a(x,C)\le\pi/2-\varepsilon$, $d_a(x,D)\le\pi/2-\varepsilon$. Then
$$ \cos d_a(\Pi_C(x),\Pi_D(x)) \ge 1-c_\varepsilon \delta_a(C,D)$$
with $c_\varepsilon:= 2(\pi^{-1}+\tan(\pi/2-\varepsilon))$.
\end{lemma}

\begin{proof}
Put $\delta_a(C,D)=\delta_a$. If $x\in C\cup D$, say $x\in C$, then $d_a(\Pi_C(x),\Pi_D(x)) = d_a(x,\Pi_D(x))\le \delta_a$ and hence
$$ \cos d_a(\Pi_C(x),\Pi_D(x)) \ge \cos \delta_a \ge 1-\frac{2}{\pi}\delta_a\ge 1-c_\varepsilon\delta_a,$$
as asserted. We can, therefore, assume in the following that $x\notin C\cup D$.

We have $\Pi_C(x)\not=o$, since $d_a(x,C)<\pi/2$. Set $p:= \Pi_C(x)/\|\Pi_C(x)\|$. Similarly, $\Pi_D(x)\not=o$, and we set $q:= \Pi_D(x)/\|\Pi_D(x)\|$. Let $e:= \Pi_{C^\circ}(x)/\|\Pi_{C^\circ}(x)\|$ (we have $\Pi_{C^\circ}(x)\not=o$, since $x\notin C$).

We can assume that there is a two-dimensional unit sphere ${\mathbb S}^2$ containing the three unit vectors $p,q,e$ (${\mathbb S}^2\subseteq\Sd$ if $d\ge 3$; and if $d=2$, we embed $\R^2$ in $\R^3$). We have $x_0:=x/\|x\|\in{\mathbb S}^2$, since $x$ is a linear combination of $p$ and $e$. Put $d_a(x,C)=d_a$ (this is also the spherical distance of $x_0$ and $p$). The spherical cap in ${\mathbb S}^2$ with center $x_0$ and spherical radius $d_a+\delta_a$ contains $q$ (since $d_a(x,\Pi_C(x))= d_a$ and there is a point $z\in D$ with $d_a(z,\Pi_C(x))\le \delta_a$). On the other hand, $d_a(q,C)\le \delta_a$ since $q\in D$. By (\ref{2.1}), $d_a(q,C)+d_a(q,C^\circ)=\pi/2$. Since $e\in C^\circ$, we get
$$ d_a(q,e)\ge d_a(q,C^\circ) =\frac{\pi}{2}-d_a(q,C) \ge \frac{\pi}{2} -\delta_a.$$
Therefore, the largest possible angular distance between $p$ and $q$ is attained if 
$$ q= w\cos\delta_a+ e\sin\delta_a \quad \mbox{with } w\in e^\perp\cap{\mathbb S}^2$$
and 
$$ \langle x_0,q\rangle = \cos(d_a+\delta_a).$$
If $q$ satisfies these equations, then
\begin{eqnarray*}
\cos(d_a+\delta_a) &=& \langle x_0,q\rangle = \langle p\cos d_a+e\sin d_a, w\cos\delta_a +e\sin\delta_a\rangle\\
&=& \langle p,w\rangle \cos d_a\cos\delta_a+\sin d_a\sin\delta_a.
\end{eqnarray*}
Consequently,
\begin{eqnarray*} 
\cos d_a(p,q) &=&\langle p,q\rangle =\langle p,w\rangle \cos\delta_a\\
&=& \frac{1}{\cos d_a}[\cos(d_a+\delta_a)-\sin d_a\sin\delta_a]\\
&=& \cos\delta_a-2(\tan d_a)\sin\delta_a\\
&\ge& 1-\frac{2}{\pi}\delta_a  -2(\tan d_a)\delta_a.
\end{eqnarray*}
From this, the assertion follows. 
\end{proof}

\section{The `Master Steiner formula', localized}\label{sec3}

First we consider polyhedral cones. Let $\cP\C^d\subset\C^d$ denote the subset of polyhedral cones, and $\cP\C^d_*=\cP\C^d\cap\C^d_*$. For $P\in\cP\C^d$ and $k\in\{0,\dots,d\}$, we denote by $\F_k(C)$ the set of $k$-dimensional faces of $C$ (which are again polyhedral cones). The $k$-skeleton of $C$ is defined by 
$$ {\rm skel}_k(C)= \bigcup_{F\in\F(C)} {\rm relint}\,F,$$
where ${\rm relint}$ denotes the relative interior. 

We denote by $\B(X)$ the $\sigma$-algebra of Borel sets in a topological space $X$ and define the `conic $\sigma$-algebra'
$$ \widehat\B(\Rd) = \{ A\in\B(\Rd): \lambda a \in A \mbox{ for $a\in A$ and $\lambda>0$}\}$$
and the `biconic $\sigma$-algebra' (terminology from \cite{AB12}, see also \cite{Ame15})
$$ \widehat\B(\Rd\times\Rd) = \{\eta\in\B(\Rd\times\Rd): (\lambda x,\mu y)\in \eta\mbox{ for $(x,y)\in\eta$ and $\lambda,\mu>0$}\},$$
where $\Rd\times\Rd$ is endowed with the product topology. Clearly, these are $\sigma$-algebras. We denote by $\bf g$
a standard Gaussian random vector in $\R^d$, and $\BP$ denotes probability. In the following form, the definition of conic support measures was suggested in \cite[(1.25)]{Ame15}.

\begin{definition}
The {\em conic support measures} $\Omega_0(C,\cdot),\dots,\Omega_d(C,\cdot)$ of the polyhedral cone $C$ are defined by
$$ \Omega_k(C,\eta) = \BP\{\Pi_C({\bf g})\in{\rm skel}_kC,\; (\Pi_C({\bf g}),\Pi_{C^\circ}({\bf g}))\in\eta\},\quad\eta\in \widehat\B(\Rd\times\Rd).$$
\end{definition}

It is clear that $ \Omega_k(C,\cdot)$ is a measure. Its total measure is the $k$th conic intrinsic volume, denoted by 
$v_k(C)=\Omega_k(C,\Rd\times\Rd)$.

First we deal with polyhedral cones, and for these we extend the `Master Steiner formula' of McCoy and Tropp \cite{MT14b} from conic intrinsic volumes to conic support measures. Let $f:\R_+^2\to\R_+$ be a measurable function. Let $\bE$ denote expectation and ${\mathbbm 1}_\eta$ the indicator function of $\eta$. For a cone $C\in\PC^d$ we define a measure $\varphi_f(C,\cdot)$ by
$$ \varphi_f(C,\eta):= \bE\left[ f(\|\Pi_C(\g)\|^2,\|\Pi_{C^\circ}(\g)\|^2)\cdot {\mathbbm 1}_\eta(\Pi_C(\g),\Pi_{C^\circ}(\g))\right]$$
for $\eta\in \widehat \B (\R^d\times\R^d)$. The following theorem expresses this measure as a linear combination of conic support measures, with coefficients depending on the function $f$.             
\begin{theorem}\label{T3.1}
Let $C\in\PC^d$ be a polyhedral cone. Let $f:\R_+^2\to\R_+$ be a measurable function such that $\varphi_f(C,\eta)$ is finite for all $\eta\in \widehat\B(\R^d\times\R^d)$. Then
\begin{equation}\label{3.1} 
\varphi_f(C,\eta) = \sum_{k=0}^d I_k(f)\cdot\Omega_k(C,\eta)
\end{equation}
for $\eta\in \widehat\B(\R^d\times\R^d)$, where the coefficients are given by
\begin{equation}\label{3.2} 
I_k(f)=\varphi_f(L_k,\R^d\times\R^d) 
\end{equation}
with an (arbitrary) $k$-dimensional subspace $L_k$ of $\Rd$. Explicitly,
$$ I_k(f) = \frac{\omega_k\omega_{d-k}}{\sqrt{2\pi}^{\hspace{1pt}d}} \int_0^\infty \int_0^\infty f(r^2,s^2)  e^{-\frac{1}{2}(r^2+s^2)} r^{k-1}s^{d-k-1}\,\D s \,\D r$$
for $k=1,\dots,d-1$ and
\begin{eqnarray*}
I_0(f) &=& \frac{\omega_d}{\sqrt{2\pi}^{\hspace{1pt}d}} \int_0^\infty f(0,s^2)  e^{-\frac{1}{2}s^2} s^{d-1}\,\D s,\\
I_d(f) &=& \frac{\omega_d}{\sqrt{2\pi}^{\hspace{1pt}d}} \int_0^\infty f(r^2,0)  e^{-\frac{1}{2}r^2} r^{d-1}\,\D r.
\end{eqnarray*}
\end{theorem}

The proof need not be carried out here, since it proceeds by an obvious modification of the proof given by McCoy and Tropp \cite{MT14b} for the global case (that is, for $\eta=\Rd\times\Rd$). Using the notation of \cite{MT14b}, the modification consists in replacing, in the proof of \cite[Lemma 8.1]{MT14b}, the function ${\mathbbm 1}_{{\rm relint}(F)}(\bf u)$ by ${\mathbbm 1}_{{\rm relint}(F)}(\bf u){\mathbbm 1}_\eta(\bf u,\bf w)$. The continuity of the function $f$ that is assumed in \cite[Lemma 8.1]{MT14b} is not needed at this stage.

By specialization, we obtain the local spherical Steiner formula, in a conic version. For this, let $C\in\PC^d_*$ and $\eta\in\widehat\B(\R^d\times\R^d)$. For $0\le \lambda<\pi/2$, we define the {\em angular local parallel set}
\begin{equation}\label{2.1.10b} 
M^a_\lambda(C,\eta):= \{x\in\R^d: d_a(x,C)\le \lambda,\; (\Pi_C(x),\Pi_{C^\circ}(x))\in\eta\}.
\end{equation}
This is a Borel set. By $\gamma_d$ we denote the standard Gaussian measure on $\Rd$.

\begin{theorem}\label{T3.2}
Let $C\in\PC^d_*$. The Gaussian measure of the angular local parallel set $M^a_\lambda(C,\eta)$, for $0\le \lambda<\pi/2$, is given by
\begin{equation}\label{2.1.13a}
\gamma_d(M^a_\lambda(C,\eta)) = \sum_{k=1}^d g_k(\lambda)\cdot\Omega_k(C,\eta),
\end{equation}
where
\begin{equation}\label{2.1.13b} 
g_k(\lambda) = \frac{\omega_k\omega_{d-k}}{\omega_d} \int_0^\lambda \cos^{k-1}\varphi \sin^{d-k-1}\varphi\,\D\varphi
\end{equation}
for $k=1,\dots,d-1$ and $g_d\equiv 1$.
\end{theorem}

Theorem \ref{T3.2} follows from Theorem \ref{T3.1} by choosing the function $f$ defined by
$$ f(a,b) := \left\{ \begin{array}{ll}1 & \mbox{if } a\le b\tan^2\lambda,\\ 0 &\mbox{otherwise},\end{array}\right.$$
for which
$$ \varphi_f(C,\eta)= \bP\{d_a(\g,C)\le\lambda,\; (\Pi_C(x),\Pi_{C^\circ}(x))\in\eta\} = \gamma_d(M^a_\lambda(C,\eta)).$$
Theorem \ref{T3.2} is equivalent to the local spherical Steiner formula, proved by Glasauer \cite[Satz 3.1.1]{Gla95}.

We need weak convergence to extend the conic support measures to general convex cones. For a function $f:\Sd\times\Sd\to\R$, we define its homogeneous extension by 
$$ f_h(x,y) := f\left(\frac{x}{\|x\|},\frac{y}{\|y\|}\right) \quad\mbox{for } x,y \in\R^d\setminus \{o\}.$$
For finite measures $\mu,\mu_i$ on $\widehat\B(\R^d\times\R^d)$, the weak convergence $\mu_i\xrightarrow{w} \mu$ is defined by each of the following equivalent conditions (for the equivalence see, e.g., \cite[p. 176]{Ash72}):\\[1mm]
$\rm (a)$ For all continuous functions $f:\Sd\times\Sd\to\R$,
$$ \int_{\R^d\times\R^d} f_h(x,y)\,\mu_i(d(x,y)) \to \int_{\R^d\times\R^d} f_h(x,y)\,\mu(d(x,y))$$
as $i\to\infty$;\\[1mm]
$\rm (b)$ For every open set $\eta\subset\R^d\times\R^d$,
$$ \mu(\eta) \le \liminf_{i\to\infty} \mu_i(\eta),$$
and
$$ \mu(\R^d\times\R^d)= \lim_{i\to\infty} \mu_i(\R^d\times\R^d).$$

If we now define
$$ \mu_\lambda(C,\eta) = \gamma_d(M^a_\lambda(C,\eta)),\quad \eta\in \widehat\B(\Rd\times\Rd)$$
for given $C\in\C^d$ and $\lambda\ge 0$, then $\mu_\lambda(C,\cdot)$ is a finite measure on $\B(\Rd\times\Rd)$, and $\mu_\lambda$ depends weakly continuously on $C$. The counterpart to this fact for Euclidean convex bodies is \cite[Thm. 4.1.1]{Sch14}, and Glasauer \cite[Hilfssatz 3.1.3]{Gla95} has carried this over to spherical space. His argument can easily be adapted to the conic setting. Convergence of cones refers to the angular Hausdorff metric. 

\begin{lemma}\label{L3.2}
Let $0\le\lambda<\pi/2$. Let $C,C_i\in \C^d_*$ ($i\in\N$) be cones with $C_i\to C$ as $i\to\infty$. Then $\mu_\lambda(C_i,\cdot) \xrightarrow{w} \mu_\lambda(C,\cdot)$.
\end{lemma}

With the aid of this lemma, the following theorem can be proved. It is deliberately formulated close to \cite[Thm. 4.2.1]{Sch14}, to show the analogy. 

\begin{theorem}\label{T3.3}
To every cone $C\in\C^d_*$ there exist finite positive measures $\Omega_0(C,\cdot),\dots,\Omega_d(C,\cdot)$ on the $\sigma$-algebra $\widehat\B(\R^d\times\R^d)$ such that, for every $\eta\in\widehat\B(\R^d\times\R^d)$ and every $\lambda$ with $0\le \lambda<\pi/2$, the Gaussian measure of the angular local parallel set $M^a_\lambda(C,\eta)$ is given by
\begin{equation}\label{2.2.4a} 
\mu_\lambda(C,\eta) = \sum_{k=1}^d g_k(\lambda)\cdot\Omega_k(C,\eta)
\end{equation}
with
\begin{equation}\label{2.2.5} 
g_k(\lambda) = \frac{\omega_k\omega_{d-k}}{\omega_d} \int_0^\lambda \cos^{k-1}\varphi \sin^{d-k-1}\varphi\,\D\varphi
\end{equation}
for $k=1,\dots,d-1$ and $g_d\equiv 1$.

The mapping $C\mapsto \Omega_k(C,\cdot)$ (from $\C^d_*$ into the space of finite measures on $\widehat\B(\R^d\times\R^d)$) is a weakly continuous valuation. 

For each $\eta\in  \widehat\B(\R^d\times\R^d)$, the function $\Omega_k(\cdot,\eta)$ (from $\C^d_*$ to $\R$) is measurable.
\end{theorem}

The proof just requires to reformulate Glasauer's \cite{Gla95} arguments in the conical setting.

\vspace{3mm}

Our aim is to extend Theorem \ref{T3.1} to general convex cones. For this, we need the following lemma.
\begin{lemma}\label{L3.3}
Let $f:\R_+^2\to\R_+$ be continuous and bounded. Let $C,C_i\in\C_*^d$ ($i\in\N$) be cones with $C_i\to C$ as $i\to \infty$. Then $\varphi_f(C_i,\cdot) \xrightarrow{w} \varphi_f(C,\cdot)$.
\end{lemma}

\begin{proof}
For $x\in\R^d$ we have
$$ \Pi_{C_i}(x) \to \Pi_{C}(x),\qquad \Pi_{C^\circ_i}(x) \to \Pi_{C^\circ}(x) \qquad (i\to\infty),$$
which follows from Lemma \ref{L2.1a}. This implies
\begin{equation}\label{2.2.13a}  
\lim_{i\to\infty} f(\|\Pi_{C_i}(x)\|^2,\|\Pi_{C_i^\circ}(x)\|^2) = f(\|\Pi_C(x)\|^2,\|\Pi_{C^\circ}(x)\|^2).
\end{equation}
Let $\eta\subset\R^d\times\R^d$ be open. If $(\Pi_{C}(x),\Pi_{C^\circ}(x))\in \eta$, we have $(\Pi_{C_i}(x),\Pi_{C_i^\circ}(x))\in \eta$ for almost all $i$. Thus,
$$ {\mathbbm 1}_\eta(\Pi_C(x),\Pi_{C^\circ}(x))\le {\mathbbm 1}_\eta(\Pi_{C_i}(x),\Pi_{C_i^\circ}(x))$$
for almost all $i$. We deduce that
\begin{eqnarray*}
&& f(\|\Pi_{C}(x)\|^2,\|\Pi_{C^\circ}(x)\|^2) {\mathbbm 1}_\eta(\Pi_{C}(x),\Pi_{C^\circ}(x))\\
&& \le \liminf_{i\to\infty}
f(\|\Pi_{C_i}(x)\|^2,\|\Pi_{C_i^\circ}(x)\|^2) {\mathbbm 1}_\eta(\Pi_{C_i}(x),\Pi_{C_i^\circ}(x)).
\end{eqnarray*}
Fatou's lemma shows that
$$ \varphi_f(C,\eta) \le  \liminf_{i\to\infty} \varphi_f(C_i,\eta).$$
Further, (\ref{2.2.13a}) together with the dominated convergence theorem gives
$$ \lim_{i\to\infty} \varphi_f(C_i,\R^d\times\R^d)= \varphi_f(C,\R^d\times\R^d).$$
Both relations together yield the assertion. 
\end{proof}

We can now extend Theorem \ref{T3.1} to general convex cones.

\begin{theorem}\label{T3.4}
Let $C\in\C^d$ be a convex cone. Let $f:\R_+^2\to\R_+$ be a measurable function such that 
$$ \varphi_f(C,\eta):= \bE\left[ f(\|\Pi_C(\g)\|^2,\|\Pi_{C^\circ}(\g)\|^2)\cdot {\mathbbm 1}_\eta(\Pi_C(\g),\Pi_{C^\circ}(\g))\right]$$
is finite for all $\eta\in \widehat\B(\R^d\times\R^d)$. Then
\begin{equation}\label{3.10} 
\varphi_f(C,\eta) = \sum_{k=0}^d I_k(f)\cdot\Omega_k(C,\eta)
\end{equation}
for $\eta\in \widehat\B(\R^d\times\R^d)$, where the coefficients are as in Theorem $\rm \ref{T3.1}$.
\end{theorem}

\begin{proof}
For given $C\in\C^d$, there is a sequence of polyhedral cones $C_i\in\PC^d$ converging to $C$, and by Theorem \ref{T3.1} we have
$$\varphi_f(C_i,\cdot) = \sum_{k=0}^d I_k(f)\cdot\Omega_k(C_i,\cdot)$$
for $i\in\N$, where $I_k(f)=\varphi_f(L_k,\R^d\times\R^d)$. We assume first that $f$ is continuous. Then $\varphi_f(C_i,\cdot) \xrightarrow{w} \varphi_f(C,\cdot)$ by Lemma \ref{L3.3}, and $\Omega_k(C_i,\cdot) \xrightarrow{w} \Omega_k(C,\cdot)$ by Theorem \ref{T3.3}. We conclude that
\begin{equation}\label{2.2.16}
\varphi_f(C,\cdot) = \sum_{k=0}^d I_k(f)\cdot\Omega_k(C,\cdot).
\end{equation}

We modify the argumentation of McCoy and Tropp \cite{MT14b}. We fix a Borel set $\eta\in\widehat\B(\R^d\times\R^d)$. Let $h:\R_+^2\to\R$ be bounded and continuous. We decompose $h=h^+-h^-$ with bounded continuous functions $h^+,h^-:\R_+^2\to\R_+$ and define $\varphi_h(C,\eta)=\varphi_{h^+}(C,\eta)-\varphi_{h^-}(C,\eta)$. Then we can write
\begin{eqnarray*}
\varphi_h(C,\eta) &=& \bE \left[h(\|\Pi_C({\bf g})\|^2,\|\Pi_{C^\circ}({\bf g})\|^2)\cdot{\mathbbm 1}_\eta(\Pi_C({\bf g}),\Pi_{C^\circ}({\bf g}))\right]\\
&=& \int_{B\eta} h(\|\Pi_C(x)\|^2,\|\Pi_{C^\circ}(x)\|^2)\,\gamma_d(\D x)\\
&=& \int_{\R_+^2} h(s,t)\,\mu_\eta(\D(s,t)),
\end{eqnarray*}
where
$$ B_\eta:= \{x\in\R^d: (\Pi_C(x),\Pi_{C^\circ}(x))\in\eta\}$$
and where $\mu_\eta$ is the pushforward of the restriction $\gamma_d\fed B_\eta$ under the mapping $x\mapsto (\|\Pi_C(x)\|^2,\|\Pi_{C^\circ}(x)\|^2)$.
Denoting by $\mu_k$ the pushforward of $\gamma_d$ under the mapping $x\mapsto (\|\Pi_{L_k}(x)\|^2,\|\Pi_{L_k^\circ}(x)\|^2)$ (where $L_k$ is the subspace used in (\ref{3.2})), we have
$$ \sum_{k=0}^d \varphi_h(L_k,\R^d\times\R^d)\cdot\Omega_k(C,\eta) = \sum_{k=0}^d\left(\int_{\R_+^2} h(s,t)\,\mu_k(\D(s,t))\right)\Omega_k(C,\eta).$$
Therefore, (\ref{2.2.16}) gives
$$ \int_{\R_+^2} h\,\D\mu_\eta =\int_{\R_+^2} h\,\D\left(\sum_{k=0}^d \Omega_k(C,\eta)\mu_k\right).$$
Since this holds for all bounded, continuous real functions $h$ on $\R_+^2$, it follows (e.g., from \cite[Lemma 30.14]{Bau90}) that
$$ \mu_\eta= \sum_{k=0}^d \Omega_k(C,\eta)\mu_k.$$
Integrating a nonnegative, measurable function $f$ on $\R_+^2$ with respect to these measures gives the assertion.
\end{proof}

\section{H\"older continuity of the conic support measures}\label{sec4}

The fact that the conic support measures are weakly continuous will now be improved, by establishing H\"older continuity with respect to a metric which metrizes the weak convergence. This is in analogy to the case of support measures of convex bodies, which was treated in \cite{HS15}. 

On $\R^d\times\R^d$, we use the standard Euclidean norm $\|\cdot\|$, which is thus also defined on $\Sd\times\Sd$. For a bounded real function $f$ on $\Sd\times\Sd$, we define
$$ \|f\|_L :=\sup_{a\not= b} \frac{|f(a)-f(b)|}{\|a-b\|},\qquad \|f\|_\infty:= \sup_a|f(a)|.$$
Let $\F_{bL}$ be the set of all functions $f:\Sd\times\Sd\to\R$ with $\|f\|_{L}\le 1$ and $\|f\|_\infty\le 1$. The functions in $\F_{bL}$ are continuous and hence integrable with respect to every finite Borel measure on $\Sd\times\Sd$.  We define the {\em bounded Lipschitz distance} of finite Borel measures $\mu,\nu$ on $\Sd\times\Sd$ by
$$ d_{bL}(\mu,\nu):= \sup \left\{\left| \int_{\Sd\times\Sd} f \,\D \mu - \int_{\Sd\times\Sd} f\,\D\nu \right|: f\in \F_{bL}\right\}.$$
This defines a metric $d_{bL}$. Convergence of a sequence of finite Borel measures on $\Sd\times\Sd$ with respect to this metric is equivalent to weak convergence of the sequence (see, e.g., \cite[Sect. 11.3]{Dud02}). 
 For finite measures $\mu,\nu$ on the biconic $\sigma$-algebra $\widehat\B(\R^d\times\R^d)$, we then define
$$ d_{bL}(\mu,\nu):= \sup \left\{\left| \int_{\R^d\times\R^d} f_h \,\D \mu - \int_{\R^d\times\R^d} f_h\,\D\nu \right|: f\in \F_{bL}\right\}.$$
Clearly, this defines a metric $d_{bL}$ which metrizes the weak convergence of finite measures on the biconic $\sigma$-algebra. 

In the following theorem, the assumption $\delta_a(C,D)\le 1$ is convenient, but could be relaxed to $\delta_a(C,D)\le \pi/2-\varepsilon$, with some $\varepsilon>0$. The constant $c$  would then also depend on $\varepsilon$.

\begin{theorem}\label{T4.1}
Let $C,D\in \C^d_*$ be convex cones, and suppose that $\delta_a(C,D)\le 1$. Then
$$ d_{bL}(\Omega_k(C,\cdot),\Omega_k(D,\cdot)) \le c\,\delta_a(C,D)^{1/2}$$
for $k\in\{1,\dots,d-1\}$, where $c$ is a constant depending only on $d$.
\end{theorem}

\begin{proof}
Let $C\in\C^d_*$, $\eta\in\widehat\B(\R^d\times\R^d)$, and $0\le \lambda < \pi/2$. In contrast to the earlier definitions of $M^a_\lambda(C,\eta)$ and $C_\lambda^a$, we now need
$$ M^a_{0,\lambda}(C,\eta):= \{x\in\R^d: 0< d_a(x,C) \le \lambda,\, (\Pi_C(x),\Pi_{C^\circ}(x))\in\eta\}$$
and
$$ C^a_{0,\lambda}:= C_\lambda^a \setminus C = \{x\in\R^d : 0< d_a(x,C) \le \lambda\},$$
further
$$ \nu_\lambda(C,\eta):= \gamma_d( M^a_{0,\lambda}(C,\eta)).$$

Let $F_\lambda:C^a_{0,\lambda} \to \R^d \times \R^d$ be defined by $F_\lambda(x) := (\Pi_C(x),\Pi_{C^\circ}(x))$ for $x\in C^a_{0,\lambda}$. Then $F_\lambda$ is continuous, and for $\eta\in\widehat\B(\R^d\times\R^d)$ we have $\gamma_d(F_\lambda^{-1}(\eta)) = \nu_\lambda(C,\eta)$. Therefore, for any $\nu_\lambda(C,\cdot)$-integrable, homogeneous function $f_h$ on $\R^d\times\R^d$,
$$ \int_{\R^d\times\R^d} f_h\,\D\nu_\lambda(C,\cdot) = \int_{\C^a_{0,\lambda}} f_h\circ(\Pi_C,\Pi_{C^\circ})\,\D\gamma_d.$$

Now we assume that also $D\in\C^d_*$ and that $\delta_a(C,D)\le 1$. Let $f:\Sd\times \Sd\to\R$ be a function with $\|f\|_L\le 1$ and $\|f\|_\infty\le 1$. Then
\begin{eqnarray*}
&& \left| \int_{\R^d\times\R^d} f_h\,\D\nu_\lambda(C,\cdot) -  \int_{\R^d\times\R^d} f_h\,\D\nu_\lambda(D,\cdot) \right|\\
&&= \left| \int_{C^a_{0,\lambda}} f_h\circ(\Pi_C,\Pi_{C^\circ})\,\D\gamma_d - \int_{D^a_{0,\lambda}} f_h\circ(\Pi_D,\Pi_{D^\circ})\,\D\gamma_d \right|\\
&& \le \int_{C^a_{0,\lambda}\cap D^a_{0,\lambda}} \left| f_h\circ(\Pi_C,\Pi_{C^\circ}) -  f_h\circ(\Pi_D,\Pi_{D^\circ})\right|\D\gamma_d\\
&& \hspace{4mm} +  \int_{C^a_{0,\lambda}\setminus D^a_{0,\lambda}} \left| f_h\circ(\Pi_C,\Pi_{C^\circ})\right|\D\gamma_d +  \int_{D^a_{0,\lambda}\setminus C^a_{0,\lambda}} \left| f_h\circ(\Pi_D,\Pi_{D^\circ})\right|\D\gamma_d \\
&& \le  \int_{C^a_{0,\lambda}\cap D^a_{0,\lambda}} \left| f_h\circ(\Pi_C,\Pi_{C^\circ}) -  f_h\circ(\Pi_D,\Pi_{D^\circ})\right|\D\gamma_d + \gamma_d(C^a_{0,\lambda}\setminus D^a_{0,\lambda}) + \gamma_d(D^a_{0,\lambda}\setminus C^a_{0,\lambda}).
\end{eqnarray*}
Let $x\in C^a_{0,\lambda}\cup D^a_{0,\lambda}$, $x\not= o$, and write
$$ \frac{\Pi_C(x)}{\|\Pi_C(x)\|}=u,\quad \frac{\Pi_{C^\circ}(x)}{\|\Pi_{C^\circ}(x)\|}= u^\circ,\quad \frac{\Pi_D(x)}{\|\Pi_D(x)\|}=v, \quad \frac{\Pi_{D^\circ}(x)}{\|\Pi_{D^\circ}(x)\|}=v^\circ.$$
By the homogeneity of $f_h$ and the Lipschitz property of $f$, we have
\begin{eqnarray*}
&& \left| f_h\circ(\Pi_C,\Pi_{C^\circ}) -  f_h\circ(\Pi_D,\Pi_{D^\circ})\right|(x)\\
&&= \left| f_h(\Pi_C(x),\Pi_{C^\circ}(x)) -  f_h(\Pi_D(x),\Pi_{D^\circ}(x))\right|\\
&&= |f(u,u^\circ) -f(v,v^\circ)|\\
&& \le \|(u,u^\circ) - (v,v^\circ)\|  \le \|u-v\| +\|u^\circ-v^\circ\|\\
&&= 2\sin\frac{1}{2}d_a(u,v) + 2\sin\frac{1}{2}d_a(u^\circ,v^\circ)\\
&&= 2\sin\frac{1}{2}d_a(\Pi_C(x),\Pi_D(x)) + 2\sin\frac{1}{2}d_a(\Pi_{C^\circ}(x),\Pi_{D^\circ}(x)),
\end{eqnarray*}
where $ \|u-v\|= 2\sin\frac{1}{2}d_a(u,v)$ for unit vectors $u,v$ was used. By Lemma \ref{L2.2},
$$ 2\sin^2\frac{1}{2}d_a(\Pi_C(x),\Pi_D(x)) = 1- \cos d_a(\Pi_C(x),\Pi_D(x)) \le c_\varepsilon \delta_a(C,D),$$
with $\varepsilon =(\pi/2)-1$. By Lemma \ref{L2.1}, we have $\delta_a(C^\circ,D^\circ) =\delta_a(C,D)$, hence also
$$  2\sin^2\frac{1}{2}d_a(\Pi_{C\circ}(x),\Pi_{D^\circ}(x))  \le c_\varepsilon \delta_a(C,D).$$
This yields
$$ \int_{C^a_{0,\lambda}\cap D^a_{0,\lambda}} \left| f_h\circ(\Pi_C,\Pi_{C^\circ}) - f_h\circ(\Pi_D,\Pi_{D^\circ})\right|\D\gamma_d \le b \delta_a(C,D)^{1/2}$$
with $b = 2\sqrt{2c_\varepsilon}$.

Write $\delta_a(C,D)=:\delta$. To estimate $\gamma_d(C^a_{0,\lambda}\setminus D^a_{0,\lambda})$, let $x\in C^a_{0,\lambda}\setminus D^a_{0,\lambda}$. Then $x\in C_\lambda^a\setminus C$ and $x\notin D_\lambda^a\setminus D$. If $x\in D$, then $d_a(x,C)\le\delta$, hence $x\in C_\delta^a\setminus C$. If $x\notin D$, then $x\notin D_\lambda^a$, but $x\in C_\lambda^a\subseteq (D_\delta^a)_\lambda^a\subseteq D_{\lambda+\delta}^a$, thus $x\in D_{\lambda+\delta}^a\setminus D_\lambda^a$. It follows that
$$ C^a_{0,\lambda} \setminus D^a_{0,\lambda} \subseteq (C_\delta^a\setminus C) \cup (D_{\lambda+\delta}^a\setminus D_\lambda^a).$$
Therefore,
$$ \gamma_d(C^a_{0,\lambda}\setminus D^a_{0,\lambda}) \le \gamma_d(C_\delta^a)-\gamma_d(C) +\gamma_d(D_{\lambda+\delta}^a)-\gamma_d(D_\lambda^a).$$
By (\ref{2.2.4a}) (with $\eta=\R^d\times\R^d$),
$$ \gamma_d(C_\delta^a)-\gamma_d(C) = \sum_{k=1}^{d-1} g_k(\delta)v_k(C) \le \sum_{k=1}^{d-1}\frac{\omega_k\omega_{d-k}}{\omega_d}\cdot \delta=: C(d)\delta,$$
where (\ref{2.2.5}) and $v_k\le 1$ were observed. Similarly, $\gamma_d(D_{\lambda+\delta}^a)-\gamma_d(D_\lambda^a)\le C(d)\delta$. Here $C$ and $D$ can be interchanged, and we obtain
$$ \gamma_d(C^a_{0,\lambda}\setminus D^a_{0,\lambda}) + \gamma_d(D^a_{0,\lambda}\setminus C^a_{0,\lambda}) \le 4 C(d)\delta.$$
Altogether, we have obtained
$$ \left| \int_{\R^d\times\R^d} f_h\,\D\nu_\lambda(C,\cdot) -  \int_{\R^d\times\R^d} f_h\,\D\nu_\lambda(D,\cdot) \right| \le b\delta_a(C,D)^{1/2}+  4 C(d)\delta_a(C,D).$$
Since this holds for all functions $f:\Sd\times \Sd\to\R$ with $\|f\|_L\le 1$ and $\|f\|_\infty\le 1$, and since $\delta_a(C,D)\le 1$, we conclude that
$$ d_{bL}(\nu_\lambda(C,\cdot),\nu_\lambda(D,\cdot)) \le b_1\delta_a(C,d)^{1/2},$$
with a constant $b_1$ depending only on $d$.

By Theorem \ref{T3.3} (and observing that $M_{0,\lambda}^a(C,\eta)\cap C=\emptyset$), we have
$$ \nu_\lambda(C,\eta) = \sum_{k=1}^{d-1} g_k(\lambda)\cdot\Omega_k(C,\eta).$$
Since the coefficient functions $g_1,\dots,g_{d-1}$ are linearly independent on $[0,1]$, we can choose numbers $0<\lambda_1<\dots<\lambda_{d-1} < 1$ and numbers $a_{ij}$, depending only on $i$ and $j$, such that
$$ \Omega_i(C,\cdot) =\sum_{i=1}^{d-1} a_{ij}\nu_{\lambda_j}(C,\cdot)\quad\mbox{for } i=1,\dots,d-1$$
(details are in Glasauer's \cite{Gla95} proof of the spherical version of Theorem \ref{T3.3}). Using the definition of the bounded Lipschitz metric, we now obtain
$$ d_{bL}(\Omega_i(C,\cdot),\Omega_i(D,\cdot)) \le \sum_{i=1}^{d-1} |a_{ij}|d_{bL}\left(\nu_{\lambda_j}(C,\cdot),\nu_{\lambda_j}(D,\cdot)\right) \le c\,\delta_a(C,D)^{1/2}$$
with a suitable constant $c$. This completes the proof. 
\end{proof}

\noindent Author's address:\\[2mm]Rolf Schneider\\Mathematisches Institut, Albert-Ludwigs-Universit{\"a}t\\D-79104 Freiburg i. Br., Germany\\E-mail: rolf.schneider@math.uni-freiburg.de

\end{document}